\documentclass[12pt,reqno]{amsart}
\usepackage[active]{srcltx}
\usepackage{latexsym}
\usepackage{amsmath,amssymb,amsthm,mathtools}
\usepackage{mathrsfs,bm,color}
\usepackage{hyperref}

\newtheorem{theo}{{\sc Theorem}}[section]

\newtheorem{maintheo}{{\sc Theorem}}

\newtheorem{mainprop}{{\sc Proposition}}

\newtheorem{lem}[theo]{{\sc Lemma}}

\newcommand{\E}{\mathbb{E}}
\newcommand{\hcal}{\mathcal{H}}
\newcommand{\N}{\mathbb{N}}
\newcommand{\ncal}{\mathcal{N}}
\newcommand{\R}{\mathbb{R}}
\newcommand{\Z}{\mathbb{Z}}

\newcommand{\T}{I_N}

\newcommand{\ad}{\operatorname{ad}}
\newcommand{\Herm}{\operatorname{Herm}}
\newcommand{\U}{\operatorname{U}}

\newcommand{\Psib}{\bm{\Psi}}
\renewcommand{\P}{\mathbb{P}}

\newcommand{\B}[2]{\langle B Y_{N}^{#1}, Y_N^{#2} \rangle}
\newcommand{\A}[2]{\langle A Y_{N}^{#1}, Y_N^{#2} \rangle}
\renewcommand{\u}[2]{u_{N,{#1}}(#2)}
\newcommand{\ubar}[2]{\overline{u_{N,{#1}}(#2)}}

\let\temp\phi
\let\phi\varphi
\let\varphi\temp

\renewcommand{\epsilon}{\varepsilon}

\def\Xint#1{\mathchoice
{\XXint\displaystyle\textstyle{#1}}%
{\XXint\textstyle\scriptstyle{#1}}%
{\XXint\scriptstyle\scriptscriptstyle{#1}}%
{\XXint\scriptscriptstyle\scriptscriptstyle{#1}}%
\!\int}
\def\XXint#1#2#3{{\setbox0=\hbox{$#1{#2#3}{\int}$ }
\vcenter{\hbox{$#2#3$ }}\kern-.6\wd0}}

\def\dashint{\Xint-}

\renewcommand{\Re}{\operatorname{Re}}
\renewcommand{\Im}{\operatorname{Im}}

\renewcommand{\u}[2]{u_{N,{#1}}(#2)}
\newcommand{\X}{X_N}

\newcommand{\W}{W_{n,N}}

\def\cprime{$'$}

\allowdisplaybreaks

\title[Quantum ergodicity of Wigner induced random spherical harmonics]
{Quantum ergodicity of Wigner induced random spherical harmonics}

\author{Robert Chang}
\address{Department of Mathematics, Northwestern  University, Evanston, IL 60208, USA}

\email{hchang@math.northwestern.edu}

\date{\today}

\begin{document}

\begin{abstract}
We introduce a new notion of a `random orthonormal basis of spherical harmonics' of $L^2(S^2)$ using generalized Wigner ensembles and show that such a random basis is almost surely quantum ergodic. Similar quantum ergodicity results (with varying degrees of generality) are obtained in \cite{Z1,Z2,Z3,M,BL} for random Laplacian eigenfunctions defined using Haar measures on unitary groups. Our main contribution comes from the use of a more general measure than previously studied, as the Gaussian unitary ensemble (which induces Haar measure on the unitary group) is a special case of the generalized Wigner ensemble. We are able to work with this more general class of measures because Wigner eigenvectors are asymptotically Gaussian, a result proved in \cite{KY, TV} (with additional assumptions on the moments) and \cite{BY}. Our quantum ergodicity statement also provides a semi-classical realization of the probabilistic `local quantum unique ergodicity' of \cite{BY}.

\end{abstract}

\maketitle

\section{Introduction}\label{sec:intro}

Let $(M,g)$ be a compact Riemannian manifold. Let $\Delta = \Delta_g$ be the Laplace-Beltrami operator and consider the eigenvalue problem $(\Delta - \lambda_k)\phi_k = 0$ with $0 < \lambda_1 \le \lambda_2 \le \dotsb \uparrow \infty$. The eigenfunctions $\phi_k$ are said to be quantum ergodic if for every pseudo-differential operator $A \in \Psi^0(M)$ of degree zero, we have

\begin{equation}\label{eqn:qe}
\lim_{\lambda \rightarrow \infty} \frac{1}{\#\{\lambda_k \le \lambda\}}\sum_{\lambda_k \le \lambda} \left\lvert \langle A \phi_k, \phi_k \rangle - \omega(A) \right\rvert^2 = 0,
\end{equation}
where

\begin{equation}\label{eqn:omega}
\omega(A) := \int_{S^*M} \sigma_A \,d\mu_L
\end{equation}
is the integral of the principal symbol $\sigma_A$ of $A$ with respect to the normalized Liouville measure $\mu_L$ on the cosphere bundle $S^*M$. A fundamental result that explains how the mixing properties of a classical system is reflected in the microlocal properties of eigenfunctions is the quantum ergodicity theorem of Shnirelman \cite{S}, Zelditch \cite{Z87}, and Colin de Verdi\`ere \cite{CdV}. The theorem states that if the geodesic flow is ergodic, then the Laplacian eigenfunctions $\phi_k$ enjoy the quantum ergodic property \eqref{eqn:qe}. In particular, modulo a density zero subsequence, the eigenfunctions become delocalized in phase space in the sense that

\begin{equation}\label{eqn:qe 2}
\langle A \phi_{k_j}, \phi_{k_j} \rangle \rightarrow \omega(A) \quad \text{for all $A \in \Psi^0(M)$.}
\end{equation}
The asymptotic behavior \eqref{eqn:qe 2} need not hold when the geodesic flow is no longer assumed to be ergodic. On the sphere, for instance, the geodesic flow is completely integrable and direct computations show that the standard spherical harmonics localize not only on phase space, but also on the base manifold $S^2$.

This fact notwithstanding, it is shown in \cite{Z1} that a random orthonormal basis (defined using Haar measures on unitary groups) of spherical harmonics is almost surely quantum ergodic, a result that is extended to Laplacian eigenfunctions on compact Riemannian manifolds in \cite{Z2,Z3,M,BL}. The purpose of this paper is to return to the sphere and prove quantum ergodicity for a wider class of `random' spherical harmonics. Consider the orthogonal decomposition of $L^2(S^2)$ into a direct sum of subspaces $\hcal_N = \operatorname{span}\{Y_N^k \mathrel{} \mid \mathrel{} -N \le k \le N\}$ spanned by the standard degree $N$ spherical harmonics. Here, by `standard,' we mean spherical harmonics $Y_N^k$ that are the joint eigenfunctions of the Laplacian $\Delta = \Delta_{S^2}$ and the $z$-component of the angular momentum operator $L_z = \frac{1}{i}\frac{d}{d\varphi}$, that is,

\begin{equation*}
\begin{dcases}
\Delta Y_N^k = -N(N+1)Y_N^k, \\
\frac{1}{i}\frac{\partial}{\partial \varphi} Y_N^k = k Y_N^k.
\end{dcases}
\end{equation*}
We write $d_N = \dim \hcal_N = 2N+1$ for the dimension of $\hcal_N$.
 
Let $H_N \in \Herm(d_N)$ be a generalized Wigner matrix. (See Section~\ref{sec:random matrix} for background on random matrix theory.) For $-N \le k \le N$, let $u_{N,k} = (u_{N,k}(\alpha))_{\alpha = -N}^N$ be the eigenvectors of $H_N$. Our object of study is the Wigner induced random basis $\{\psi_{N,k}\}_{k = -N}^N$ for $\hcal_N$ obtained by `transplanting the Wigner eigenvectors onto the sphere' in the obvious way:

\begin{equation}\label{eqn:random onb}
\psi_{N,k} := \sum_{\alpha = -N}^N u_{N,k}(\alpha) Y_N^\alpha, \quad -N \le k \le N.
\end{equation}
An equivalent way of thinking about the random basis $\{\psi_{N,k}\}$ is to identify it with a unitary change-of-basis matrix $U_N = (u_{N,k}(\alpha))_{-N \le k,\alpha \le N}$ viewed as an element of the probability space $(\U(d_N),\mu_N)$. The probability measure $\mu_N$ on the unitary group $\U(d_N)$ is induced by a generalized Wigner matrix in the following way. Let $\pi$ be the map from Hermitian matrices to unitary matrices modulo the maximal torus $\U(1)^{d_N}$ defined by

\begin{equation*}
\pi \colon \Herm(d_N) \rightarrow \U(d_N)/\U(1)^{d_N}, \quad H_N = U_N^*D(\bm{\lambda})U_N \mapsto [U_N],
\end{equation*}
where $U_N$ is a unitary matrix that diagonalizes $H_N$ and $D(\bm{\lambda})$ is the resulting diagonal matrix. If we write $\mu_N^\mathrm{W}$ for the measure on the Hermitian matrices that describes the generalized Wigner ensemble, then the induced measure $\mu_N$ on the unitary group is simply the pushforward of $\mu_N^{\mathrm{W}}$ under the above map $\pi$, that is,

\begin{equation} \label{eqn:induced measure}
\mu_N := \pi_* \mu_N^\mathrm{W}.
\end{equation}

The construction of a Wigner induced random basis \eqref{eqn:random onb} for the finite dimensional subspace $\hcal_N$ extends naturally to all of $L^2(S^2)$. Indeed, let $U$ be the operator that acts block-diagonally on the decomposition $L^2(S^2) = \bigoplus_{N \ge 0} \hcal_N$ so that the restrictions $U|_{\hcal_N} = U_N \in \U(d_N)$ to the subspaces yield a sequence of independent unitary matrices of the appropriate dimensions. By the preceding paragraph, a Wigner induced random orthonormal basis $\Psib = \{\psi_{N,k}\}_{-N \le k \le N, N \ge 0}$ for all of $L^2(S^2)$ may be identified with such an operator $U$ viewed as an element of the product probability space $\prod_{n \ge 0}(\U(d_N),\mu_N)$. Henceforth, when the context is clear, we will refer to $\Psib$ simply as a `random basis' with the understanding that it is constructed randomly with respect to the product measure $\prod \mu_N$.

For technical reasons, certain indices $k$ need to be excluded from our computations. Let $0 < \nu < \frac{3}{4}$ be a positive constant (guaranteed by Theorem~\ref{theo:asymptotic normality}), and let

\begin{equation}\label{eqn:indices}
\T  = [[-N, -N + N^{1/4}]] \cup [[-N + N^{1 - \nu}, N - N^{1 - \nu}]] \cup [[N - N^{1/4}, N]]
\end{equation}
be the subset of indices $-N \le k \le N$ that are, in the random matrix theory language, `in the bulk' and `near the edges.' We can only work with indices belonging to $\T$ because the asymptotic normality result of Bourgade-Yau (Theorem~\ref{theo:asymptotic normality}), which we rely on, is established only for $k \in \T$. (The set $\T$ displayed above is precisely the set $\mathbb{T}_N$ in the statement of Theorem~1.2 in the original paper \cite{BY}, except that the our indexing convention is $k \in [-N,N]$, and the convention of \cite{BY} is $k \in [1,N]$.) It is expected that Theorem~\ref{theo:asymptotic normality} holds for all indices $k$ (see the remark immediately following Definition~5.1 in \cite{BY}). Luckily, the set $I_N$ is sufficient for deriving a quantum ergodicity statement because we are still left with a density one subsequence after discarding indices in the intermediate regime, that is,

\begin{equation*}
\frac{\lvert \{k \in \T \}\rvert}{\lvert \{k \in [-N,N]\} \rvert} \rightarrow 1.
\end{equation*}

Given a pseudo-differential operator $A \in \Psi^0(M)$ of order zero and a random basis $\Psib$, let $\X = X_N^A(\{\psi_{N,k}\}) \colon (\U(d_N), \mu_N) \rightarrow \R_{\ge 0}$ be random variables given by

\begin{equation}\label{eqn:rv}
\X  = X_N^A(\{\psi_{N,k}\}) = \frac{1}{d_N} \sum_{k \in \T} \lvert\langle A\psi_{N,k}, \psi_{N,k}\rangle - \omega(A) \rvert^2,
\end{equation}
where $\omega(A)$ is defined in \eqref{eqn:omega}. Even though the random variable \eqref{eqn:rv} depends on the choice of a pseudo-differential operator and a random basis, for notational simplicity we will continue to write $\X := X_N^A(\{\psi_{N,k}\})$. Our quantum ergodicity result is formulated in terms of $\X$.

\begin{maintheo}\label{theo:main thm}
Let $\Psib$ be a Wigner induced random orthonormal basis of spherical harmonics for $L^2(S^2)$. Then $\Psib$ is almost surely quantum ergodic with respect to the product probability measure $\prod \mu_N$ in the sense that

\begin{equation*}
\lim_{M \rightarrow \infty} \frac{1}{M}\sum_{N = 0}^M \X = 0 \quad{\text{a.s.}}
\end{equation*}
for every $A \in \Psi^0(S^2)$.
\end{maintheo}

Note that the random variables $\X$ are independent by construction. Theorem~\ref{theo:main thm} is therefore an easy consequence of the Kolmogorov convergence criterion and Strong Law of Large Numbers once we show that $\E \X \rightarrow 0$ and $\E \X^2$ is bounded. Indeed, the following holds.

\begin{maintheo}\label{theo:main thm 2}
We have $\E \X = O(d_N^{-\epsilon_0})$ and $\E \X^2 = O(d_N^{-\epsilon'_0})$ for some $\epsilon_0, \epsilon'_0 > 0$ guaranteed by Theorem~\ref{theo:asymptotic normality}.
\end{maintheo}

This is a good place for some remarks. First, since we only work with random spherical harmonics in this paper, we confine ourselves to describing the construction of random bases on $S^2$. A similar construction that involves partitioning the spectrum of the Laplacian appropriately can be used to make sense of random bases (defined using either Haar measures or Wigner induced measures on unitary groups) on any compact Riemannian manifold. Readers are referred to \cite{Z2,Z3,M,BL} for the general construction. A natural next step is to extend our quantum ergodicity result to Wigner induced random bases of Laplacian eigenfunctions or approximate eigenfunctions on other manifolds.

Second, it is known that the eigenvectors of a Gaussian unitary ensemble is distributed by Haar measure on the unitary group. Since the generalized Wigner ensembles contain GUE as a special case, the measure with respect to which Wigner eigenvectors are distributed (i.e., the Wigner induced measure $\mu_N$) is a vast generalization of Haar measure.  It is unknown to the author if such measures can be given an explicit characterization. Nevertheless, universality results from random matrix theory are robust enough for showing that Wigner induced random bases enjoy the same quantum ergodicity property as `GUE induced random bases' (i.e., random bases defined using Haar measure) on the sphere.

Finally, the methods presented in this paper can be used to prove quantum ergodicity of Wigner induced random spherical harmonics on higher dimensional spheres $S^p$ for any $p \ge 2$. It will be clear from the proof that $\epsilon_0$ and $\epsilon'_0$ in the statement of Theorem~\ref{theo:main thm 2} are independent of the dimension $p$ because, in the notation of Theorem~\ref{theo:asymptotic normality}, we have $\epsilon_0 = \epsilon_0(Q_1)$ and $\epsilon'_0 = \epsilon'_0(Q_2)$ where $Q_1,Q_2$ are polynomials of the form

\begin{equation*}
Q_1(z_1, z_2, z_3, z_4) = z_1 z_2 \overline{z}_3 \overline{z}_4 \quad \text{and} \quad Q_2(z_1, \dotsc, z_8) = z_1 z_2z_3 z_4 \overline{z}_5 \overline{z}_6\overline{z}_7 \overline{z}_8.
\end{equation*}
While $\epsilon_0, \epsilon'_0$ remain fixed for all $p \ge 2$, the dimension $d_N$ of the space of degree $N$ spherical harmonics grows like $N^{p-1}$ on $S^p$. Substituting the asymptotics for $d_N$ into the statement of Theorem~\ref{theo:main thm 2} gives $\E \X = O(N^{-\epsilon_0(p - 1)})$ and $\E \X^2 = O(N^{-\epsilon'_0(p - 1)})$. Observe that, for all $p$ sufficiently large, the Borel-Cantelli lemma becomes applicable and implies the stronger convergence statement that $\X \rightarrow 0$ almost surely instead of the Ces\`aro means $\frac{1}{M}\sum_{N=0}^M X_N \rightarrow 0$.

The rest of the paper is organized as follows. Section~\ref{sec:random matrix} provides a brief summary of random matrix theory that will be used in our proofs. The key result is Theorem~\ref{theo:asymptotic normality}, which states that Wigner eigenvectors (with the appropriate scaling) are asymptotically Gaussian random variables. Section~\ref{sec:rot inv} is devoted to proving Proposition~\ref{prop:rot inv}, which is a special case of Theorem~\ref{theo:main thm 2}. The techniques developed for this special case extends easily to prove the main theorems in Section~\ref{sec:general case}.

The author expresses his sincere gratitude towards Steve Zelditch and Antonio Auffinger for their patience and generous assistance that greatly improved the manuscript.

\subsection{Asymptotic normality of Wigner eigenvectors and Bourgade-Yau local QUE}\label{sec:random matrix}

We now summarize a universality result for Wigner eigenvectors proved in \cite{BY}. In keeping with the indexing convention for spherical harmonics, the indices in this section continue to range from $-N$ to $N$. Recall also that $d_N = 2N+1$.

By a generalized Wigner matrix we mean a Hermitian matrix $H_N = (h_{jk})_{-N \le j,k \le N} \in \Herm(d_N)$ such that:

\begin{itemize}
\item The entries $h_{jk}$ are independent random variables for $j \le k$, each with mean zero and variance $\E h_{jk}^2 =: \sigma_{jk}^2$ satisfying the normalization condition $\sum_{j=-N}^N \sigma_{jk}^2 = 1$ for $k$ fixed;

\item There exists a constant $c_1 > 0$ independent of $N$ such that $(c_1N)^{-1} \le \sigma_{jk}^2 \le c_1N$ for all $-N \le j,k \le N$;

\item There exists a constant $c_2 > 0$ independent of $N$ such that $\E(\bm{h}_{jk}^*\bm{h}_{jk}) \ge c_2 N^{-1}$ in the sense of inequality between $2 \times 2$ positive matrices, where $\bm{h}_{jk} := (\Re h_{jk}, \Im h_{jk})$;

\item For any $q \in \N$, there exists a constant $C_q > 0$ such that for any $N$ and any $-N \le j,k \le N$, we have $\E\lvert \sqrt{d_N} h_{jk} \rvert^q \le C_q$.
\end{itemize}

Let $u_{N,k} = (u_{N,k}(\alpha))_{\alpha = -N}^N$ denote the eigenvectors of a generalized Wigner matrix $H_N \in \Herm(d_N)$. The eigenvectors, indexed by $k \in [-N,N]$, are ordered so that the corresponding eigenvalues form a nondecreasing sequence. Of course, an eigenvector is well-defined only up to a phase $e^{i\theta} \in \U(1)$. This phase ambiguity may be eliminated, for instance, by considering instead the equivalence class $[u_{N,k}]$.

\begin{theo}[Asymptotic normality for generalized Wigner eigenvectors, \cite{BY} Corollary~1.3]\label{theo:asymptotic normality}
Let $\{H_N\}$ be a sequence of generalized Wigner matrices. Let $\T$ be the set of indices away from the intermediate regime as defined in \eqref{eqn:indices} \textup{(}note that $\T$ depends on a parameter $\nu$\textup{)}. Then there exists $\nu > 0$ such that for any $k \in \T$  and $J \subset \{-N, \dotsc, N\}$ with $\lvert J \rvert = m$, we have

\begin{equation*}
\sqrt{d_N}(u_{N,k}(\alpha))_{\alpha \in J} \rightarrow \big(\ncal_j^{(1)} + i\ncal_j^{(2)}\big)_{j=1}^m
\end{equation*}
in the sense of convergence in moments modulo phases, where $\ncal_j^{(1)}, \ncal_j^{(2)}$ are independent standard Gaussians. More precisely, for any polynomial $Q$ in $2m$ variables, there exists $\epsilon = \epsilon(Q) > 0$ such that for sufficiently large $N$ we have

\begin{align*}\notag
\sup_{\substack{J \subset \{-N, \dotsc, N\}\\ \lvert J \rvert = m,\; k \in \T}} &\left\lvert \E Q\left(\sqrt{2N}\big(e^{i\omega}u_{N,k}(\alpha), e^{-i\omega}\overline{u_{N,k}(\alpha)}\big)_{\alpha \in J}\right)\right.\\
&\quad\quad \quad\quad - \left.\E Q\left(\big(\ncal_j^{(1)} +i \ncal_j^{(2)}, \ncal_j^{(1)} - i\ncal_j^{(2)}\big)_{j=1}^m \right)\right\rvert \le d_N^{-\epsilon}.
\end{align*}
Here $\omega$ a phase independent of $H_N$ and uniform on $(0,2\pi)$.
\end{theo}

In fact, a stronger statement is proved Theorem~1.2 of \cite{BY}, namely the projection $\langle \bm{q}, u_{N,k}\rangle$ of an eigenvector to any unit vector $\bm{q} \in \R^{d_N}$ is asymptotically normal. As a corollary, generalized Wigner eigenvectors are `locally quantum unique ergodic' in the following sense. Let $a_N \colon \{-N, \dotsc, N\} \rightarrow [-1,1]$ be a function with $\sum_{\alpha=-N}^N a_N(\alpha) = 0$ and let $\lvert a_N \rvert = \#\{-N \le \alpha \le N \mathrel{} \mid \mathrel{} a_N(\alpha) \neq 0\}$ be the size of its support.

\begin{theo}[Local QUE for generalized Wigner eigenvectors, \cite{BY} Corollary~1.4]\label{theo:local QUE}
Let $\{H_N\}$ be a sequence of generalized Wigner matrices. Then there exists $\epsilon > 0$ such that for any $\delta > 0$, there exists a constant $C > 0$ so that for every sequence of functions $\{a_N\}$ as above and $k \in I_N$ we have

\begin{equation}\label{eqn:local QUE}
\P\left(\left\lvert\frac{d_N}{\lvert a_N\rvert} \langle a_Nu_{N,k},u_{N,k} \rangle\right\rvert > \delta\right) \le C(d_N^{-\epsilon} + \lvert a_N \rvert^{-1}),
\end{equation}
where $\langle a_Nu_{N,k},u_{N,k}\rangle := \sum_{\alpha = -N}^N a_N(\alpha)\lvert u_{N,k}(\alpha)\rvert^2$.
\end{theo}
Theorem~\ref{theo:asymptotic normality} shows that Wigner eigenvectors are asymptotically flat even on small scales by choosing the test functions $a_N$ to have small supports. Note that since the left-hand side of \eqref{eqn:local QUE} depends only the eigenvectors but not the eigenvalues, the measure used in Theorem~\ref{theo:local QUE} is precisely the induced measure $\mu_N$ defined in \eqref{eqn:induced measure}.

We take this opportunity to remark that on a compact manifold $(M,g)$, the analogue to the limiting formula \eqref{eqn:local QUE} given by

\begin{equation}\label{eqn:flat}
\int_M f(x) \lvert \phi_k(x)\rvert^2\,dx \rightarrow \int_M f(x)\,dx \quad \text{for every $f \in C^\infty(M)$}
\end{equation}
is insufficient for concluding that $\{\phi_k\}$ is quantum ergodic in the sense of \eqref{eqn:qe} or \eqref{eqn:qe 2}. This is because delocalization on the base manifold $M$ is a much weaker condition than diffuseness in the phase space $S^*M$. For instance, the Laplacian eigenfunctions $e^{i \langle \lambda, x \rangle}$ on a flat torus $\R^n / 2 \pi \Z^n$ are delocalized in the sense of \eqref{eqn:flat}. But if $\{\lambda_k\}$ is a sequence of lattice points for which the unit vectors $\lambda_k / \lvert \lambda_k \rvert$ tend to a limit vector $\xi \in \R^n$, then the asymptotic formula

\begin{equation*}
\langle A e^{i\langle \lambda_k, x \rangle}, e^{i\langle \lambda_k, x \rangle}\rangle \simeq \int_{\R^n/2 \pi \Z^n} \sigma_A\left(x, \frac{\lambda_k}{\lvert \lambda_k \rvert}\right)\,dx \quad \text{for every $A \in \Psi^0(\R^n / 2 \pi \Z^n)$}
\end{equation*}
shows that the corresponding weak* limit is a delta mass on the invariant Lagrangian torus $T_{\xi} \subset S^*M$ for the geodesic flow. Since there always exists a sequence of $\lambda_k / \lvert \lambda_k \rvert$ converging to arbitrary $\xi \in \R^n$, the eigenfunctions $e^{i \langle \lambda, x \rangle}$ are far from diffuse in phase space. Of course, in the random matrix setting it is unclear even how to interpret the phase space when the base manifold is an index set $\{-N, \dotsc, N\}$. We will need additional tools from semi-classical analysis to show that Theorem~\ref{theo:main thm} holds.

\section{Rotationally invariant case}\label{sec:rot inv}

The purpose of this section is to prove Proposition~\ref{prop:rot inv} stated below. The difference between the proposition and Theorem~\ref{theo:main thm 2} is the rotational invariance assumption we impose on $A$ (and hence on the random variable $\X$). This additional assumption allows us to isolate the key computational techniques and exhibit them in a simpler setting.

To clearly distinguish the special case we are currently considering from the general case, let us introduce some new notation. Let $B \in \Psi^0(S^2)$ denote pseudo-differential operators of degree zero that are invariant under $z$-axis rotations. To these rotationally invariant operators we associate random variables

\begin{equation}\label{eqn:inv rv}
Z_N = Z_N^B(\{\psi_{N,k}\}) = \frac{1}{d_N} \sum_{k \in I_N} \lvert \langle B \psi_{N,k}, \psi_{N,k} \rangle - \omega(B) \rvert^2,
\end{equation}
where $I_N$ is defined in \eqref{eqn:indices} and $\omega(B)$ is defined in \eqref{eqn:omega}. Our goal is to show the following.

\begin{mainprop}\label{prop:rot inv}
In the above notation, we have $\E Z_N = O(d_N^{-\epsilon})$ and $\E Z_N^2 = O(d_N^{-\epsilon'})$ for some $\epsilon, \epsilon' > 0$ guaranteed by Theorem~\ref{theo:asymptotic normality}.
\end{mainprop}

\begin{proof}[Proof of Proposition~\ref{prop:rot inv}]
Note that the rotational invariance hypothesis implies that the matrix elements $\B{\alpha}{\beta}$ vanish whenever $\alpha \neq \beta$. Rewriting the random basis elements $\psi_{N,k}$ in terms of spherical harmonics $Y_N^\alpha$ using \eqref{eqn:random onb}, the expression \eqref{eqn:inv rv} becomes

\begin{align*}
Z_N &= \frac{1}{d_N} \sum_{k \in \T} \left\lvert \sum_{\alpha,\beta = -N}^N \B{\alpha}{\beta} \u{k}{\alpha}\u{k}{\beta} - \omega(B) \,\right\rvert^2\\
&= \frac{1}{d_N} \sum_{k \in \T} \left\lvert \sum_{\alpha} \B{\alpha}{\alpha}\lvert \u{k}{\alpha} \rvert^2 - \omega(B) \,\right\rvert^2\\
&= S_1 + S_2,
\end{align*}
where

\begin{align*}
S_1 &= \frac{1}{d_N} \sum_{k \in \T} \sum_{\alpha,\beta} \B{\alpha}{\alpha}\B{\beta}{\beta} \lvert \u{k}{\alpha} \rvert^2 \lvert \u{k}{\beta} \rvert^2,\\
S_2 &= -\frac{2\omega(B)}{d_N} \sum_{k \in \T} \sum_{\alpha} \B{\alpha}{\alpha}\lvert \u{k}{\alpha}\rvert^2 + \frac{1}{d_N} \sum_{k \in \T} \omega(B)^2.
\end{align*}

We use the Weingarten formula \cite{W} to compute the expectation $\E Z_N = \E S_1 + \E S_2$. Let $(\u{k}{\alpha})_{-N \le k, \alpha \le N} \in \U(d_N)$ be a unitary matrix and for $1 \le j \le m$, let $k_j, k_j',\alpha_j, \alpha_j' \in [-N,N]$ be indices. The Weingarten formula states that the integral

\begin{equation*}
I_N(m) := \int_{\U(d_N)} \u{k_1}{\alpha_1} \dotsm \u{k_m}{\alpha_m} \ubar{k_1'}{\alpha_1'} \dotsb \ubar{k_m'}{\alpha_m'} \,dU_N
\end{equation*}
of a polynomial in the entries of $(\u{k}{\alpha})$ with respect to Haar measure $dU_N$ has an asymptotic formula in terms of the Kronecker delta functions on the indices:
\begin{equation}\label{eqn:old Weingarten}
I_N(m) = d_N^{-m}\sum \delta_{k_1 k'_{j_1}} \delta_{\alpha_1 \alpha'_{j_1}}\dotsm \delta_{k_\ell k'_{j_m}}\delta_{\alpha_\ell \alpha'_{j_m}} + O(d_N^{-m - 1}),
\end{equation}
where the sum is over all choices of $j_1, \dotsc, j_m$ as a permutation of $1, \dotsc, m$. Let $Q$ be the polynomial in $2m$ variables defined by $Q\big((z_j,w_j)_{j=1}^m\big) := z_1 \dotsm z_m \overline{w}_1 \dotsm \overline{w}_m$. Then, in the notation of Theorem~\ref{theo:asymptotic normality}, direct computation with Gaussian random variables shows that

\begin{equation}\label{eqn:comparison with Gaussian}
\left\lvert \frac{1}{d_N^m}\E Q\left( \big(\ncal_j^{(1)} + i \ncal_J^{(2)}, \ncal_J^{(1)} - i\ncal_J^{(2)}\big)_{j=1}^m\right) - I_N(m) \right\rvert = O(d_N^{-m-1})
\end{equation}
Putting together \eqref{eqn:old Weingarten}, \eqref{eqn:comparison with Gaussian}, and Theorem~\ref{theo:asymptotic normality} proves the following key lemma.

\begin{lem}\label{eqn:wigner weingarten}
Let $(\u{k}{\alpha}) \in \U(d_N)$ be a unitary matrix. For indices $k_1, \dotsc, k_m, k_1', \dotsc, k_m' \in \T$ and $\alpha_1, \dotsc, \alpha_m, \alpha_1', \dotsc, \alpha_m' \in [-N,N]$, we have

\begin{align}\label{eqn:Weingarten}\notag
& \E \left(\u{k_1}{\alpha_1} \dotsm \u{k_m}{\alpha_m}\overline{\u{k'_1}{\alpha'_1}}\dotsm \overline{\u{k'_m}{\alpha'_m}}\right) \\
&\qquad\qquad\qquad = d_N^{-m} \sum \delta_{k_1 k'_{j_1}} \delta_{\alpha_1 \alpha'_{j_1}}\dotsm \delta_{k_m k'_{j_m}}\delta_{\alpha_m \alpha'_{j_m}} + O(d_N^{-m-\epsilon})
\end{align}
for some $\epsilon = \epsilon(Q) > 0$ guaranteed by Theorem~\ref{theo:asymptotic normality}.
\end{lem}

Returning to the quantity $\E Z_N = \E S_1 + \E S_2$, we find that \eqref{eqn:Weingarten} implies
\begin{equation*}
\E \left(\lvert \u{k}{\alpha} \rvert^2 \lvert \u{k}{\beta} \rvert^2\right) = d_N^{-2} (1 + \delta_{\alpha\beta}) + O(d_N^{-2-\epsilon_1}) \quad \text{for $k \in \T$,}
\end{equation*}
which gives

\begin{align}\label{eqn:S1} \notag
\E S_1 &= \frac{1}{d_N} \sum_{k \in \T} \sum_{\alpha,\beta} \B{\alpha}{\alpha} \B{\beta}{\beta}\E \left(\lvert \u{k}{\alpha} \rvert^2 \lvert \u{k}{\beta} \rvert^2\right)\\ \notag
&= \sum_{\alpha,\beta} \B{\alpha}{\alpha} \B{\beta}{\beta} \left(\frac{1}{d_N^2}(1 + \delta_{\alpha\beta}) + O(d_N^{-2-\epsilon_1})\right)\\ 
&= \left(\frac{1}{d_N}\sum_{\alpha} \B{\alpha}{\alpha}\right)^2 + \frac{1}{d_N^2} \sum_\alpha \B{\alpha}{\alpha}^2 + O(d_N^{-\epsilon_1}).
\end{align}
The first sum in \eqref{eqn:S1} can be rewritten using semi-classical analysis. Let $\Pi_N \colon L^2(S^2) \rightarrow \hcal_N$ denote the spectral projection onto the eigenspace of degree $N$ spherical harmonics. Let $A \in \Psi^0(S^2)$ be any pseudo-differential operator of degree zero (not necessarily rotationally invariant), then Weyl's law states that

\begin{equation}\label{eqn:weyl}
\frac{1}{d_N}\sum_{\alpha} \A{\alpha}{\alpha} = \frac{1}{d_N}\operatorname{tr}(\Pi_N A \Pi_N) = \omega(A) + O(d_N^{-1}).
\end{equation}
%
%
%
For the second sum in \eqref{eqn:S1}, it suffices to note that the squares $\A{\alpha}{\alpha}^2$ of the matrix elements are uniformly bounded in $N$ because the pseudo-differential operator $A \in \Psi^0(S^2)$ (again, not necessarily rotationally invariant) is a bounded operator from $L^2(S^2)$ to itself. Since we are summing over $-N \le \alpha \le N$ (i.e., summing $d_N$ number of terms) and dividing by $d_N^2$, the second sum has only a lower order contribution:

\begin{equation}\label{eqn:square}
\frac{1}{d_N^2} \sum_\alpha \B{\alpha}{\alpha}^2 = O(d_N^{-1}).
\end{equation}
Combining \eqref{eqn:S1}, \eqref{eqn:weyl}, and \eqref{eqn:square} yields
\begin{equation*}
\E S_1 = \left(\omega(B) + O(d_N^{-1})\right)^2 + O(d_N^{-1}) + O(d_N^{-\epsilon_1}) = \omega(B)^2 + O(d_N^{-\epsilon_1}).
\end{equation*}

The asymptotics for $\E S_2$ is similarly computed. By \eqref{eqn:Weingarten}, we have

\begin{equation*}
\E \lvert \u{k}{\alpha} \rvert^2 = d_N^{-1} + O(d_N^{-1-\epsilon_2}) \quad \text{for $k \in \T$,}
\end{equation*}
whence
\begin{align*}
\E S_2 &= -\frac{2\omega(B)}{d_N}  \sum_{k \in \T} \sum_\alpha\B{\alpha}{\alpha} \E \lvert \u{k}{\alpha} \rvert^2+ \frac{1}{d_N} \sum_{k \in \T} \omega(B)^2\\
&= -2\omega(B)\sum_\alpha \B{\alpha}{\alpha}\left(\frac{1}{d_N} + O(d_N^{-1- \epsilon_2})\right) +  \omega(B)^2\\
&= -2\omega(B)^2 + \omega(B)^2 + O(d_N^{-\epsilon_2}),
\end{align*}
where the last equality follows from Weyl's law \eqref{eqn:weyl}. Adding together the expressions for $\E S_1$ and $\E S_2$ shows that $\E Z_N = O(d_N^{-\min\{\epsilon_1,\epsilon_2\}}) = O(d_N^{-\epsilon})$ as the factors of $\omega(B)^2$ cancel exactly. This proves the first part of Proposition~\ref{prop:rot inv}.

The computations for the second moment $\E Z_N^2$ is more tedious,  but no new techniques are required. Write a second copy of the random variable $Z_N$ with the indices $j, \eta, \xi$ in place of $k, \alpha,\beta$, then direct computation shows

\begin{equation*}
\E Z_N^2 = T_1 + T_2 + \dotsb + T_5,
\end{equation*}
where

\begin{align*}
T_1 &= \frac{1}{d_N^2} \sum_{k,j \in \T} \sum_{\alpha,\beta,\eta,\xi} \B{\alpha}{\alpha} \B{\beta}{\beta} \lvert \u{k}{\alpha} \rvert^2 \lvert \u{k}{\beta} \rvert^2\\
&\qquad \qquad \qquad \qquad \qquad \qquad \times \B{\eta}{\eta} \B{\xi}{\xi} \lvert \u{j}{\eta} \rvert^2 \lvert \u{j}{\xi} \rvert^2,\\
T_2 &= 	-\frac{4\omega(B)}{d_N^2} \sum_{k,j \in \T} \sum_{\alpha,\beta,\eta} \B{\alpha}{\alpha} \B{\beta}{\beta}\lvert \u{k}{\alpha} \rvert^2 \lvert \u{k}{\beta}\rvert^2\\
& \qquad \qquad \qquad \qquad \qquad \qquad\times \B{\eta}{\eta}  \lvert \u{j}{\eta} \rvert^2,\\
T_3 &= \frac{2 \omega(B)^2}{d_N^2} \sum_{k,j \in \T} \sum_{\alpha,\beta}  \B{\alpha}{\alpha}\B{\beta}{\beta} \lvert \u{k}{\alpha} \rvert^2 \lvert \u{k}{\beta} \rvert^2,\\
T_4 &= \frac{4\omega(B)^2}{d_N^2} \sum_{k,j \in \T}\sum_{\alpha,\eta} \B{\alpha}{\alpha} \lvert \u{k}{\alpha} \rvert^2 \B{\eta}{\eta} \lvert \u{j}{\eta} \rvert^2,\\
T_5 &= -\frac{4\omega(B)^3}{d_N^2} \sum_{k,j \in \T} \sum_\alpha \B{\alpha}{\alpha} \lvert \u{k}{\alpha} \rvert^2 + \frac{1}{d_N^2} \sum_{k,j \in \T} \omega(B)^4.
\end{align*}

We work out the asymptotics for $\E T_1$ in detail. Appealing once again to \eqref{eqn:Weingarten}, we have

\begin{equation}\label{eqn:T1}
\E \big(\lvert \u{k}{\alpha} \rvert^2 \lvert \u{k}{\beta} \rvert^2 \lvert \u{j}{\eta} \rvert^2 \lvert \u{j}{\xi} \rvert^2\big) = d_N^{-4} \big(C_1 +  \delta_{kj} C_2\big) + O(d_N^{-4-\epsilon_1'}),
\end{equation}
where

\begin{align*}
C_1 = C_1(\alpha,\beta,\eta,\xi) &= (1 + \delta_{\alpha\beta})(1 + \delta_{\eta\xi}),\\
C_2 = C_2(\alpha,\beta,\eta,\xi) &= \delta_{\alpha\eta}(1 + \delta_{\beta\xi} + 2\delta_{\eta\xi}) + \delta_{\alpha\xi}(1 + \delta_{\beta\eta} + 2\delta_{\beta\xi})\\
&\qquad\qquad +\delta_{\beta\eta}(1 + 2\delta_{\alpha\beta}) + \delta_{\beta\xi}(1 + 2\delta_{\eta\xi}) + 6\delta_{\alpha\beta}\delta_{\beta\xi}\delta_{\eta\xi}.
\end{align*}
These imply

\begin{align}\label{eqn:T1 2} \notag
\E T_1 &= \frac{1}{d_N^4}\sum_{\alpha,\beta,\eta,\xi} C_1(\alpha,\beta,\eta,\xi) \B{\alpha}{\alpha} \B{\beta}{\beta}\B{\eta}{\eta}\B{\xi}{\xi}\\
 & \quad +\frac{1}{d_N^5}\sum_{\alpha,\beta,\eta,\xi} C_2(\alpha,\beta,\eta,\xi) \B{\alpha}{\alpha} \B{\beta}{\beta}\B{\eta}{\eta}\B{\xi}{\xi}+ O(d_N^{-\epsilon_1'}).
\end{align}
Notice that the leading orders of $C_1$ and $C_2$ are different because there is a factor of $\delta_{kj}$ in front of $C_2$ but not $C_1$ in \eqref{eqn:T1}.

Consider the first line of the expression \eqref{eqn:T1 2} (i.e., the part that involves only $C_1$). Recall that $C_1 = (1 + \delta_{\alpha\beta})(1 + \delta_{\eta\xi}) = 1 + \delta_{\alpha\beta} + \delta_{\eta\xi} + \delta_{\alpha\beta}\delta_{\eta\xi}$ contains four terms. We claim that only the constant term has a top order contribution when computing the asymptotics of $\E T_1$; the other three terms containing Kronecker delta functions all have lower order contributions. Indeed, notice that

\begin{equation*}
\frac{1}{d_N^4}\sum_{\alpha,\beta,\eta,\xi}\delta_{\alpha\beta} \B{\alpha}{\alpha} \B{\beta}{\beta}\B{\eta}{\eta}\B{\xi}{\xi}
\end{equation*}
is equal to
\begin{equation*}
\frac{1}{d_N^4} \sum_{\alpha,\eta,\xi} \B{\alpha}{\alpha}^2\B{\eta}{\eta}\B{\xi}{\xi} = O(d_N^{-1}),
\end{equation*}
which is a lower order term because we are summing $d_N^3$ number of uniformly bounded products of matrix elements but dividing by $d_N^4$.

We now turn our attention to the second line of the expression \eqref{eqn:T1 2} (i.e., the part that involves only $C_2$). Notice that each term of $C_2$ contains at least one Kronecker delta function on the indices $\alpha,\beta,\eta,\xi$. At the same time, we are dividing the sum by $d_N^5$. Therefore, the entire second line is of order at most $O(d_N^{-2})$. These observations imply that the expected value of $T_1$ has the simple asymptotics

\begin{align*}
\E T_1 &= \frac{1}{d_N^4} \sum_{\alpha,\beta,\eta,\xi} \B{\alpha}{\alpha} \B{\beta}{\beta}\B{\eta}{\eta}\B{\xi}{\xi} + O(d_N^{-\epsilon_1'})\\\
& = \omega(B)^4 + O(d_N^{-\epsilon_1'}).
\end{align*}

Similar arguments show that

\begin{align*}
\E T_2 &= - \frac{4\omega(B)}{d_N^2} \sum_{k,j \in \T} \sum_{\alpha,\beta,\eta} \B{\alpha}{\alpha}\B{\beta}{\beta}\B{\eta}{\eta}\\
&\qquad \qquad \qquad \qquad \qquad \times \frac{1}{d_N^3}\big(1 + \delta_{\alpha\beta} + \delta_{kj}(\delta_{\alpha\eta} + \delta_{\beta\eta} + 2\delta_{\alpha\beta}\delta_{\beta\eta})\big) + O(d_N^{-\epsilon_2'})\\
&= -4\omega(B)^4 + O(d_N^{-\epsilon_2'}),\\
\E T_3 & = \frac{2\omega(B)^2}{d_N^2} \sum_{k,j \in \T}\sum_{\alpha,\beta}\B{\alpha}{\alpha}\B{\beta}{\beta} \frac{1}{d_N^2}(1 + \delta_{\alpha\beta}) + O(d_N^{-\epsilon_3'})\\
&= 2\omega(B)^4+ O(d_N^{-\epsilon_3'}),\\
\E T_4 &= \frac{4 \omega(B)^2}{d_N^2} \sum_{k,j \in \T}\sum_{\alpha,\eta}\B{\alpha}{\alpha}\B{\eta}{\eta} \frac{1}{d_N^2}(1 + \delta_{kj}\delta_{\alpha\eta})+ O(d_N^{-\epsilon_3'})\\
&= 4\omega(B)^4 +  O(d_N^{-\epsilon_3'}),\\
\E T_5 &= -\frac{4\omega(B)^3}{d_N^2} \sum_{k,j \in \T}\sum_\alpha \B{\alpha}{\alpha} \frac{1}{d_N} + \frac{1}{d_N^2} \sum_{k,j \in \T} \omega(B)^4 + O(d_N^{-\epsilon_4'})\\
&= -4\omega(B)^4 + \omega(B)^4+ O(d_N^{-\epsilon_4'}).
\end{align*}
As before, the factors of $\omega(B)^4$ cancel exactly, and we are left with

\begin{equation*}
\E Z_N^2 = \E T_1 + \dotsb + \E T_5 = O(d_N^{-\min\{\epsilon_1', \dotsc, \epsilon_4'\}}) = O(d_N^{-\epsilon'}).
\end{equation*}
This concludes the proof of Proposition~\ref{prop:rot inv}.
\end{proof}

\section{Proof of main theorems}\label{sec:general case}
We now return to Theorem~\ref{theo:main thm} and Theorem~\ref{theo:main thm 2}, which do not have invariance assumptions on the operator $A \in \Psi^0(S^2)$. This means that we can no longer assume \emph{a priori} (as we did in the previous section) that the matrix elements $\A{\alpha}{\beta}$ vanish for $\alpha \neq \beta$. We will show, however, that by taking a Fourier series representation of the operator $A$ and using orthogonality properties of the spherical harmonics, the general case reduces to the rotationally invariant case.

\subsection{Reduction to Fourier coefficients}
The goal of this section is to obtain a Fourier series representation for a general pseudo-differential operator. Let $r_\theta$ denote rotation about the $z$-axis by angle $\theta$, that is, if we write a point $x = (\cos\tau\sin\varphi,\sin\tau \sin\varphi,\cos\varphi) \in S^2$ in spherical coordinates, then

\begin{equation*}
r_\theta(x) := (\cos(\tau - \theta)\sin\varphi,\sin(\tau - \theta)\sin\varphi, \cos \varphi).
\end{equation*}
Given $A \in \Psi^0(S^2)$, form a new operator

\begin{equation*}
A_\theta := r_\theta^*  A  r_{-\theta}^* \in \Psi^0(S^2),
\end{equation*}
where $(r_\theta^* \phi)(x) := \phi(r_\theta(x))$ for any smooth function $\phi \in C^\infty(S^2)$. For $n \in \Z$, the Fourier coefficients $\hat{A}(n)$ of $A_\theta$ are defined by

\begin{equation} \label{eqn:Fourier coeff}
\hat{A}(n) := \dashint_{S^1} e^{-in \theta} A_\theta \,d\theta\in \Psi^0(S^2).
\end{equation}
These new operators are related to the original operator $A$ in the following way.

\begin{lem}\label{lem:convergence of Fourier series}
The partial sums $\sum_{\lvert n \rvert \le  N} \hat{A}(n)$ converge in the operator norm to $A$ as $N \rightarrow \infty$.
\end{lem}

\begin{proof}[Proof of Lemma~\ref{lem:convergence of Fourier series}]
Let $D_\theta$ denote the generator of $z$-axis rotation so that $r_\theta^* = e^{-i\theta D_\theta}$. Then, since $D_\theta$ and $r_\theta^*$ commute, we have

\begin{equation*}
\frac{\partial}{\partial\theta} A_\theta = \left(\frac{\partial}{\partial\theta} r_\theta^*\right)Ar_{-\theta}^* + r_\theta^* A \left(\frac{\partial}{\partial\theta} r_{-\theta}^*\right) = \frac{1}{i}(D_\theta A_\theta - A_\theta D_\theta) = \frac{1}{i}\ad_{D_\theta}(A_\theta) \in \Psi^0(M).
\end{equation*}
This implies that the map $\theta \mapsto A_\theta$ is differentiable, and by elementary properties of convolution with the Dirichlet kernel $D_N(\theta) = \sum_{n = -N}^N e^{in\theta}$ we get uniform convergence

\begin{equation*}
\sum_{n = -N}^N \hat{A}(n) = \sum_{n = -N}^N \dashint_{S^1} e^{-in\theta}A_\theta\,d\theta = \dashint_{S^1}D_N(\theta)A_\theta\,d\theta \rightarrow A_0 = A. \qedhere
\end{equation*}
\end{proof}

\begin{lem}\label{lem:decay of Fourier coeff}
For $n \neq 0$, we have $\| \hat{A}(n) \| = O(n^{-\ell})$ for every $\ell \ge 1$. 
\end{lem}

\begin{proof}[Proof of Lemma~\ref{lem:decay of Fourier coeff}]
Integrating \eqref{eqn:Fourier coeff} by parts gives

\begin{equation*}
n \hat{A}(n) = \frac{i}{2\pi}e^{-in\theta} A_\theta\bigg|_{\theta = 0}^{2\pi} - \dashint_{S^1} e^{-in\theta} \ad_{D_\theta}(A_\theta)\,d\theta = -\dashint_{S^1} e^{-in\theta} \ad_{D_\theta}(A_\theta)\,d\theta.
\end{equation*}
It follows that integrating by parts $\ell$ times yields

\begin{equation*}
(-n)^\ell \hat{A}(n) = \dashint_{S^1} e^{-in\theta} (\ad_{D_\theta})^\ell(A_\theta)\,d\theta.
\end{equation*}
Since $(\ad_{D_\theta})^\ell(A_\theta) \in \Psi^0(S^2)$ for all $\ell \ge 1$, we conclude that $n^\ell \|\hat{A}(n)\| = O(1)$.
\end{proof}
These lemmas allow us to replace $A$ with finite sums of the form $\sum_{\lvert n \rvert \le N} \hat{A}(n)$. We record several facts about the operators $\hat{A}(n)$. First, conjugating by rotation $A \mapsto  r_\theta^* Ar_{-\theta}^* = A_\theta$ changes the principal symbol of $A$ by the canonical transformation on the cosphere bundle:

\begin{equation*}
\sigma_{A_\theta}(x,\xi) = \sigma_A(r_\theta (x),(Dr_{-\theta}(x))^{-1}\xi).
\end{equation*}
It follows from definition \eqref{eqn:Fourier coeff} of $\hat{A}(n)$ that

\begin{equation} \label{eqn:liouville integral}
\omega(\hat{A}(n)) := \int_{S^*M} \dashint_{S^1}e^{-in\theta} \sigma_A(r_\theta (x),(Dr_{-\theta}(x))^{-1}\xi)\,d\theta d\mu_L = 
\begin{cases}
\omega(A) & \text{if $n = 0$,}\\
0 &\text{if $n \neq 0$,}
\end{cases}
\end{equation}
where the latter equality follows from interchanging the order of integration and using the fact that the Liouville measure $\mu_L$ is invariant under canonical transformations.

Second, from the definition of spherical harmonics, for each fixed $n$ the matrix elements of $\hat{A}(n)$ are related to those of $A$ by the identity

\begin{equation}\label{eqn:matrix elements}
\langle \hat{A}(n) Y_N^\alpha, Y_N^\beta \rangle=
\begin{cases}
\langle A Y_N^\alpha, Y_N^{\alpha -n}\rangle &\text{if $\alpha = \beta +n$}\\
0 & \text{if $\alpha \neq \beta + n$}
\end{cases}
\quad\text{simultaneously for all $N$.}
\end{equation}
In other words, the infinite block-diagonal matrix with blocks $(\langle \hat{A}(n) Y_N^\alpha, Y_N^\beta \rangle)_{\alpha,\beta = -N}^N$ is obtained from the infinite block diagonal matrix with blocks $(\A{\alpha}{\beta})_{\alpha,\beta = -N}^N$ by replacing all the entries except those on the $n$th diagonal above (or below, depending on the sign of $n$) the main diagonal by zeros.

\subsection{Computations with Fourier coefficients}\label{sec:computations general}

Having defined Fourier coefficients $\hat{A}(n)$ and discussed their properties, we proceed to compute the expected value and second moment of the associated random variables

\begin{align*}
\W  &:= \frac{1}{d_N}\sum_{k \in \T} \lvert \langle \hat{A}(n) \psi_{N,k}, \psi_{N,k} \rangle - \omega(\hat{A}(n)) \rvert^2\\
&=
\begin{dcases}
\frac{1}{d_N} \sum_{k \in \T} \left\lvert \sum_{\alpha = -N +n}^N\A{\alpha}{\alpha}\u{k}{\alpha}\overline{\u{k}{\alpha}} - \omega(A) \right\rvert^2 &\text{if $n = 0$,}\\ 
\frac{1}{d_N} \sum_{k \in \T} \left\lvert \sum_{\alpha = -N + n}^N \A{\alpha}{\alpha-n} \u{k}{\alpha}\overline{\u{k}{\alpha -n}} \right\rvert^2 & \text{if $n \neq 0$,}
\end{dcases}
\end{align*}
where the second equality is obtained by first writing $\psi_{N,k}$ in terms of $Y_N^\alpha$ using \eqref{eqn:random onb}, and then applying \eqref{eqn:liouville integral} and \eqref{eqn:matrix elements}. We make the crucial observation that the discussion following \eqref{eqn:matrix elements} implies  the  identity

\begin{equation}\label{eqn:rv sum}
\X = \sum_{n \in \Z} \W \quad \text{for each $N = 0, 1, 2, \dotsc$.}
\end{equation}
The asymptotics for $\E\W$ and $\E\W^2$ can be easily computed.

\begin{lem}\label{lem:expected value and second moment}
For each fixed $n \in \Z$, we have $\E\W = O(d_N^{-\epsilon})$ and $\E\W^2 = O(d_N^{-\epsilon'})$ for some $\epsilon, \epsilon' > 0$ guaranteed by Theorem~\ref{theo:asymptotic normality}.
\end{lem}

\begin{proof}[Proof of Lemma~\ref{lem:expected value and second moment}]
Thanks to \eqref{eqn:matrix elements}, we recognize that $\hat{A}(0)$ is a rotationally invariant operator of the kind considered in Section~\ref{sec:rot inv}. Thus, when $n = 0$ the statement of the lemma follows from Proposition~\ref{prop:rot inv}.

When $n \neq 0$, expanding the square yields

\begin{equation*}
\W = \frac{1}{d_N} \sum_{k \in \T} \sum_{\alpha,\beta}\A{\alpha}{\alpha-n} \A{\beta}{\beta-n} \u{k}{\alpha}\u{k}{\beta}\overline{\u{k}{\alpha-n}}\overline{\u{k}{\beta-n}}.
\end{equation*}
Appealing once again to the asymptotic formula \eqref{eqn:Weingarten}, we find

\begin{equation*}
\E \left(\u{k}{\alpha}\u{k}{\beta}\overline{\u{k}{\alpha-n}}\overline{\u{k}{\beta-n}}\right) = d_N^{-2} (\delta_{\alpha,\alpha-n} \delta_{\beta,\beta-n} + \delta_{\alpha,\beta-n}\delta_{\beta,\alpha-n}) + O(d_N^{-2-\epsilon}).
\end{equation*}
Since $n \neq 0$ by hypothesis, by what is now a standard argument we conclude that all the terms in the expression of $\E\W$ that contain Kronecker delta functions are of order at most $O(d_N^{-1})$, so $\E\W = O(d_N^{-\epsilon})$.

The second moment computation is equally straightforward. Indeed, we have
\begin{align*}
\W^2 &= \frac{1}{d_N^2} \sum_{k,j \in \T} \sum_{\alpha,\beta,\eta,\xi} \A{\alpha}{\alpha -n}\A{\beta}{\beta-n}\A{\eta}{\eta -n}\A{\xi}{\xi-n}\\
&\qquad \times \u{k}{\alpha}\u{k}{\beta}\overline{\u{k}{\alpha-n}}\overline{\u{k}{\beta-n}} \u{j}{\eta}\u{j}{\xi}\overline{\u{j}{\eta-n}}\overline{\u{j}{\xi-n}}.
\end{align*}
It is easy to verify using \eqref{eqn:Weingarten} that the expected value of the product of eigenvector components is asymptotically zero because every term in the asymptotic formula contains a factor of $\delta_{\alpha,\alpha-n}$ for $n = 1, \dotsc, 4$.
\end{proof}

\subsection{Approximation argument}\label{sec:final}

We finish the computations for $\E\X$ and $\E\X^2$ by an approximation argument. 

\begin{proof}[Proof of Theorem~\ref{theo:main thm 2}]
Fix some small constant $\omega > 0$, then by \eqref{eqn:rv sum} there exists $M > 0$ such that $\sum_{\lvert n \rvert > M} \W < \omega$. Using Lemma~\ref{lem:expected value and second moment} for the asymptotics of $\E \W$ yields
\begin{equation*}
\E \X \le \E \bigg(\sum_{\lvert n \rvert \le M} \W + \omega\bigg) = \sum_{\lvert n \rvert \le M} \E \W + \omega = O(d_N^{-\epsilon}) + \omega.
\end{equation*}

The asymptotics for the second moment is similarly computed using the elementary inequality $(a_1 + \dotsb + a_m)^2 \le m(a_1^2 + \dotsb + a_m^2)$ and Lemma~\ref{lem:expected value and second moment}:

\begin{align*}
\E \X^2 &\le \E\bigg(\sum_{\lvert n \rvert \le M} \W + \omega\bigg)^2 \\
&\le (2M+1)\sum_{\lvert n \rvert \le M} \E \W^2 + 2\omega\sum_{\lvert n \rvert \le M}\E\W + \omega^2\\
& = O(d_N^{-\epsilon'}) + O(d_N^{-\epsilon}) + \omega^2.
\end{align*}
Since $\omega$ is arbitrary, Theorem~\ref{theo:main thm 2} is proved with $\epsilon_0 = \epsilon$ and $\epsilon_0' = \min\{\epsilon, \epsilon'\}$.
\end{proof}

\begin{proof}[Proof of Theorem~\ref{theo:main thm}]
Let $\sigma_N^2 := \E\X^2 - (\E\X)^2$ be the variance of the random variable $\X$. Theorem~\ref{theo:main thm} shows that the sequence $\{\X\}$ satisfies Kolmogorov's convergence criterion, that is, $\sum_{N=1}^\infty \sigma_N^2/N^2 < \infty$. We may therefore invoke the Strong Law of Large Numbers to conclude that the partial sums $\frac{1}{M}\sum_{N=0}^M \X$ converge to its expected value almost surely. But $\E\X = O(d_N^{-\epsilon})$, which implies that the expected values of the partial sums converge to zero, finishing the proof of Theorem~\ref{theo:main thm}.
\end{proof}

%

\end{document}